\documentclass[letterpaper, 11pt,  reqno]{amsart}

\usepackage[margin=1.1in,marginparwidth=1.5cm, marginparsep=0.5cm]{geometry}

\usepackage{multirow,booktabs}
\usepackage{amsmath,amssymb,amscd,amsthm,amsxtra, esint}

\usepackage{mathrsfs} 

\usepackage[implicit=true]{hyperref}

\usepackage{color} 
\usepackage{ulem}
\allowdisplaybreaks[2]

\definecolor{gr}{rgb}   {0.,   0.69,   0.23 }
\definecolor{bl}{rgb}   {0.,   0.5,   1. }
\definecolor{mg}{rgb}   {0.85,  0.,    0.85}
\definecolor{yl}{rgb}   {0.8,  0.7,   0.}
\definecolor{or}{rgb}  {0.7,0.2,0.2}

\newtheorem{theorem}{Theorem} [section]

\newtheorem{lemma}[theorem]{Lemma}
\newtheorem{proposition}[theorem]{Proposition}
\newtheorem{remark}[theorem]{Remark}

\newtheorem*{ackno}{Acknowledgements}

\DeclareMathOperator*{\supp}{supp}

\DeclareMathOperator{\Id}{Id}

\newcommand{\noi}{\noindent}
\newcommand{\Z}{\mathbb{Z}}
\newcommand{\R}{\mathbb{R}}

\newcommand{\T}{\mathbb{T}}

\newcommand{\E}{\mathbb{E}}

\newcommand{\F}{\mathcal{F}}

\newcommand{\al}{\alpha}

\newcommand{\dl}{\delta}

\newcommand{\nb}{\nabla}

\newcommand{\Dl}{\Delta}
\newcommand{\eps}{\varepsilon}
\newcommand{\kk}{\kappa}
\newcommand{\g}{\gamma}

\newcommand{\ld}{\lambda}

\newcommand{\s}{\sigma}

\newcommand{\ft}{\widehat}

\newcommand{\wt}{\widetilde}
\newcommand{\cj}{\overline}

\newcommand{\dt}{\partial_t}

\newcommand{\too}{\longrightarrow}

\newcommand{\ta}{\theta}

\renewcommand{\l}{\ell}
\renewcommand{\o}{\omega}
\renewcommand{\O}{\Omega}

\newcommand{\les}{\lesssim}
\newcommand{\ges}{\gtrsim}

\newcommand{\jb}[1]
{\langle #1 \rangle}

\newcommand{\ind}{\mathbf 1}

\newcommand{\N}{\mathbb{N}}

\numberwithin{equation}{section}
\numberwithin{theorem}{section}

\newcommand{\PP}{\mathbb{P}}

\newcommand{\z}{\zeta}

\DeclareMathOperator{\Law}{Law}

\newcommand{\W}{\mathcal{W}}

\newcommand{\dr}{\theta}
\newcommand{\Dr}{\Theta}

\newcommand{\Ha}{\mathbb{H}_a}

\newcommand{\proj}{\Pi}

\makeatletter
\@namedef{subjclassname@2020}{%
  \textup{2020} Mathematics Subject Classification}
\makeatother

\begin{document}
\baselineskip = 14pt

\title[Log-correlated Gibbs measures]
{Critical threshold for 
weakly interacting
log-correlated focusing Gibbs measures}

\author[D.~Greco, T.~Oh,  L.~Tao, and  L.~Tolomeo]
{Damiano Greco,  Tadahiro Oh,  Liying Tao, and Leonardo Tolomeo}

\address{
Damiano Greco\\School of Mathematics\\
The University of Edinburgh\\
and The Maxwell Institute for the Mathematical Sciences\\
James Clerk Maxwell Building\\
The King's Buildings\\
Peter Guthrie Tait Road\\
Edinburgh\\ 
EH9 3FD\\
 United Kingdom 
}

\email{dgreco@ed.ac.uk}


\address{
Tadahiro Oh, School of Mathematics\\
The University of Edinburgh\\
and The Maxwell Institute for the Mathematical Sciences\\
James Clerk Maxwell Building\\
The King's Buildings\\
Peter Guthrie Tait Road\\
Edinburgh\\ 
EH9 3FD\\
 United Kingdom, 
 and 
 School of Mathematics and Statistics, Beijing Institute of Technology,
Beijing 100081, China
}

\email{hiro.oh@ed.ac.uk}

\address{
Liying Tao, Graduate School of China Academy of Engineering Physics, Beijing, 100088, China}

\email{taoliying20@gscaep.ac.cn}

\address{
Leonardo Tolomeo,  School of Mathematics\\
The University of Edinburgh\\
and The Maxwell Institute for the Mathematical Sciences\\
James Clerk Maxwell Building\\
The King's Buildings\\
Peter Guthrie Tait Road\\
Edinburgh\\ 
EH9 3FD\\
United Kingdom}

\email{l.tolomeo@ed.ac.uk}



\begin{abstract}

We study 
log-correlated 
Gibbs measures
on the $d$-dimensional torus
with weakly interacting focusing quartic potentials
whose coupling constants tend to $0$ as we remove regularization.
In particular, we exhibit a phase transition for this model 
by identifying
a critical threshold, separating the weakly and strongly coupling regimes;
in the weakly coupling regime, 
we show that the frequency-truncated measures converge to 
the base Gaussian measure (possibly with a renormalized $L^2$-cutoff), 
whereas, in the strongly coupling regime, 
we prove non-convergence of 
the frequency-truncated measures, even up to a subsequence.
Our result answers an open  question posed by 
Brydges and Slade (1996).

\end{abstract}


\maketitle

\vspace{-5mm}
%

\section{Log-correlated  Gibbs measures}
\label{SEC:1}

In this paper, we study 
the Gibbs measure 
$\rho$ 
on 
the $d$-dimensional torus on $\T^d = (\R/2\pi\Z)^d$, formally given by 
\begin{align}
d\rho(u) = Z^{-1} \exp \bigg(\ld \int_{\T^d} u^4 dx\bigg) d\mu(u), 
\label{Gibbs1}
\end{align}

\noi
where 
the coupling constant
$\ld > 0$ denotes the strength of 
the focusing (i.e.~attractive) interaction
and 
$\mu$ denotes the log-correlated Gaussian free field on $\T^d$, 
formally given by 
\begin{align}\label{gauss0}
d \mu 
= Z^{-1} e^{-\frac 12 \| u\|_{H^{d/ 2} }^2    } du.
\end{align}

\noi
In particular, our interest is to study the {\it weakly interacting} case
whose meaning we will make precise in the following;
see \eqref{GibbsN1}.

Let us first introduce some notations.
Recall that the Gaussian measure $\mu$ in \eqref{gauss0} 
corresponds to 
 the induced probability measure under the map:\footnote{We endow $\T^d$  with the normalized Lebesgue measure $dx_{\T^d}=(2\pi)^{-d}dx$. With a slight abuse of notation, we still use $dx$ to denote the normalized Lebesgue measure.}
\begin{align}
\o\in \O \longmapsto u(\o) = \sum_{n \in \Z^d } \frac{ g_n(\o)}{\jb{n}^\frac d2} e_n, 
\label{map}
\end{align}

%

\noi
where 
$\jb{\,\cdot\,} = (1+|\,\cdot\,|^2)^\frac{1}{2}$, 
$e_n=e^{i n\cdot x}$, 
and $\{ g_n \}_{n \in \Z^d}$
is a sequence of mutually independent standard complex-valued
Gaussian random variables on a probability space 
$(\O,\F,\PP)$
conditioned that  $g_{-n} = \cj{g_n}$, $n \in \Z^d$.  
It is easy to see from \eqref{map} that a typical element under $\mu$
is merely a distribution, 
thus requiring
 a renormalization 
 on the interaction potential $ \ld  \int_{\T^d} u^4 dx$.
 Given $N \in \N$,  we define the frequency projector
$\pi_N$ onto the frequencies $\{|n|\le N\}$ by setting
\begin{align}\label{pi1}
\pi_N f = 
\sum_{ |n| \leq N}  \ft f (n)  e_n.
\end{align}

\noi
Note that, for each fixed $x \in \T^d$, 
$\pi_N u(x)$ is
a mean-zero real-valued Gaussian random variable with variance:
\begin{align}\label{sigma1}
\s_N = \E\big[(\pi_N u)^2(x)\big] = \sum _{|n| \le N} \frac1{\jb{n}^d}
\sim  \log N \too \infty, 
\end{align}

\noi
as $N\to\infty$. 
We then define the Wick renormalized  power
$:\! (\pi_N u)^k\!:$
by setting
\begin{align}\label{Wick1}
:\! (\pi_N u)^k (x) \!: \, \stackrel{\text{def}}{=} H_k(\pi _N u(x); \s_N), 
\end{align}

\noi
where 
$H_k(x;\s)$ is the Hermite polynomial of degree $k$
with a variance parameter $\s$. 
For readers' convenience, we write out the first few Hermite polynomials:
\begin{align}
\begin{split}
& H_0(x; \s) = 1, 
\qquad 
H_1(x; \s) = x, 
\qquad
H_2(x; \s) = x^2 - \s,   \\
& H_3(x; \s) = x^3 - 3\s x, 
\qquad 
H_4(x; \s) = x^4 - 6\s x^2 +3\s^2.
\end{split}
\label{H1a}
\end{align}

\noi
See \cite{Kuo, OST}
for further discussions.
By setting
\begin{align}
R_N(u)= \int_{\T^d} :\!(\pi_N u)^4 \!:dx, 
\label{Wick0}
\end{align}

\noi
we define the truncated Gibbs measure $\rho_N$ by 
\begin{align}
d\rho_N(u) 
= Z_N^{-1} e^{\ld_N R_N(u)} d\mu(u)
= Z_N^{-1} \exp \bigg(\ld_N \int_{\T^d} :\!(\pi_N u)^4 \!:dx\bigg) d\mu(u).
\label{Gibbs1a}
\end{align}

\noi
A standard argument shows that 
$R_N(u)$ is Cauchy in $L^p(\mu)$
for any finite $p \ge 1$;
see Lemma~\ref{LEM:conv0}.
When $\ld_N \equiv \ld < 0$
(i.e.~the defocusing case),\footnote{Here, the notation ``$\ld_N \equiv \ld$''
means that the sequence $\{\ld_N\}_{N\in \N}$
is constant, taking the value $\ld$.}
Nelson's estimate allows
us to define the defocusing log-correlated Gibbs measure $\rho$
in \eqref{Gibbs1}
as the unique limit\footnote{When $d = 2$, this corresponds to the so-called $\Phi^4_2$-measure.}
 of the truncated Gibbs measures~$\rho_N$
 in \eqref{Gibbs1a};
see \cite{Nelson, OTh}.
On the other hand, when $\ld_N \equiv \ld > 0$
(i.e.~the focusing case), 
it is known that the Gibbs measure $\rho$
is not normalizable to be a probability measure
even with a renormalized $L^2$-cutoff;
see \cite{BS, OST}.\footnote{We note that the regularization  used in \cite{BS}
is different from the frequency truncation
and is based on the approximation $(1- \Dl + \eps \Dl^2)^{-1}$ (as $\eps \to 0$)
of the covariance operator $(1-\Dl)^{-1}$ of the Gaussian free field on $\T^2$.}

In \cite[p.\,489]{BS}, Brydges and Slade proposed to study 
the limiting behavior of  the following weakly interacting truncated  Gibbs measure 
$\rho_N$
with a renormalized $L^2$-cutoff:
\begin{align}\label{GibbsN1}
d\rho_N(u)=Z_{N}^{-1}
\ind_{\{|\int_{\T^d} \, : (\pi_N u)^2: \, dx| \, \leq K_N\}}  
e^{ \ld_N R_N (u)}d\mu(u),
\end{align}

\noi
where  $R_N(u)$ is as in \eqref{Wick0}, 
by taking $\ld_N \to 0 $
and $K_N \to \infty$ as $N \to \infty$, 
and investigate the existence of 
a ``critical point, separating the weak and strong coupling regimes''.
Our main result answers this question posed by Brydges and Slade.

\begin{theorem}\label{THM:1}
Let $\{\lambda_N\}_{N\in \N}$ be 
a non-increasing sequence of positive numbers
tending to $0$ as $N \to \infty$, 
and 
let $\{K_N\}_{N\in \N}$ be a non-decreasing sequence
of positive numbers.
Then, there exist $\ld^*\ge \ld_*>0$ such that the following statements hold\textup{:}

\smallskip
\noi
\textup{(i) (weakly coupling regime).}
Suppose that 
\begin{equation}
 \ld_N\le \ld_* (K_N+\log N)^{-1}
\label{K1}
\end{equation}

\noi
for any $N \in \N$.
Then, 
given any $p\geq 1$ we have
\begin{align}
    \sup_{N\in\N}
    \Big\|
    \ind_{\{|\int_{\T^d} \, : (\pi_N u)^2: \, dx| \, \leq K_N\}}  
    e^{ \ld_N R_N (u)}
    \Big\|_{L^p(\mu)}<\infty.
\label{exp1}
\end{align}

\noi
In particular, we have
\begin{align}
    \lim_{N\rightarrow\infty}
    \ind_{\{|\int_{\T^d} \, :(\pi_N u)^2: \, dx| \, \leq K_N\}}  
    e^{\ld_N  R_N (u)}
    =
\ind_{\{|\int_{\T^d} \, :u^2: \, dx| \, \leq K\}}  
\quad     \text{in}\ L^p(\mu), 
\label{exp2}
\end{align}

\noi
where 
$K = \lim_{N \to \infty}K_N \in (0, \infty]$.
Here, 
\[\int_{\T^d} :\!u^2\!:dx = \lim_{N \to \infty}
\int_{\T^d} :\!(\pi_N u)^2\!:dx,\]

\noi
where the limit is understood in 
$L^p(\mu)$ and $\mu$-almost surely
as in Lemma \ref{LEM:conv0}.
As a consequence, 
we have

\smallskip
\begin{itemize}
\item[(i.a)] 
 If $K  =  \lim_{N \to \infty}K_N  = \infty$, 
then the truncated Gibbs measure $\rho_N$
in \eqref{GibbsN1}
converges in total variation to the base Gaussian measure $\mu$ in \eqref{gauss0}
as $N \to \infty$.

\smallskip
\item[(i.b)]
 If $K = \lim_{N \to \infty}K_N < \infty$, 
then the truncated Gibbs measure $\rho_N$
in \eqref{GibbsN1}
converges in total variation to the base Gaussian measure
with a renormalized $L^2$-cutoff\textup{:}
\begin{align*}
\ind_{\{|\int_{\T^d} \, :u^2: \, dx| \, \leq K\}}  d\mu, 
\end{align*}

\noi
as $N \to \infty$.

\end{itemize}

\smallskip
\noi
\textup{(ii) (strongly coupling regime).}
Suppose that 
\begin{equation}\label{K2}
\ld_N\ge \ld^{*}(K_N+\log N)^{-1}
 \end{equation}

\noi
for any sufficiently large $N \gg 1$.
Then, we have 
\begin{align}\label{Gibbs3}
\sup_{N\in\N}Z_{N}:=
\sup_{N\in\N}\E_{\mu}
\Big[ \ind_{\{|\int_{\T^d} \, : (\pi_N u)^2: \, dx| \, \leq K_N\}}
e^{ \ld_N R_N (u)} \Big]
=\infty.
\end{align}

\noi
As a consequence, 
the truncated Gibbs measure $\rho_N$ in \eqref{GibbsN1}
does not converge   to any  limit in total variation, even up to a subsequence.

\end{theorem}

  Theorem \ref{THM:1} in particular states
  that the weakly interacting focusing log-correlated Gibbs measure 
is trivial in the sense that, as we remove the regularization, 
we either obtain
the base Gaussian measure (possibly with a renormalized $L^2$-cutoff)
or non-normalizability\,/\,non-convergence.  
  
In a seminal work \cite{LRS}, Lebowitz, Rose, and Speer
initiated the study of focusing Gibbs measures;\footnote{More precisely, focusing $\Phi^k_d$-measures, 
where the base Gaussian measure has covariance $(1-\Dl)^{-1}$.}
see also 
\cite{BO94,BS,  CFL}.
In a series of recent works \cite{OST2, OOT, OST, OOT2}, 
the second and fourth authors
with their collaborators
completed this research program on the (non-)construction 
of focusing Gibbs measures for any dimension and any power, 
in particular, 
by treating critical models, exhibiting delicate phase transitions,
when $d = 1$ and $d = 3$.
When $d = 2$, 
there is no such phase transition 
in  the case 
 $\ld_N \equiv \ld > 0$
in~\eqref{GibbsN1}.
Theorem~\ref{THM:1}
shows that 
the weakly interacting 
model \eqref{GibbsN1}
(proposed by Brydges and Slade~\cite{BS})
is critical when $\ld_N\sim (K_N+\log N)^{-1}$, 
nicely complementing
the critical models in 
$ d= 1, 3$
studied in  \cite{OST2, OOT, OOT2}.

\begin{remark}\rm 
(i) While  we stated our result
in the real-valued setting, 
a similar result holds in the  complex-valued setting
by replacing the Hermite polynomials with 
the Laguerre polynomial; see \cite{OTh}
for a further discussion.

\smallskip

\noi
(ii) Let $d = 2$.
Consider the following weakly interacting truncated
nonlinear Schr\"odinger equation on $\T^2$:
\begin{equation}
i\dt u_N + (1-\Dl)u_N - 4\ld_N \pi_N(| \pi_N u_N |^2\pi_N  u_N)=0
\label{NLS1}
\end{equation}

\noi
with the initial data distributed by $\rho_N$ in \eqref{GibbsN1}.
A standard argument shows that $\rho_N$ is an invariant measure
for \eqref{NLS1}.
Moreover, as a dynamical consequence of Theorem \ref{THM:1}
and~\cite{BO96}, 
we see that, as $N \to \infty$, 
the solution $u_N$ to
\eqref{NLS1} 
\begin{itemize}
\item
converges to the solution $u$ to the linear Schr\"odinger equation
$i\dt u + (1-\Dl)u=0$, 
with the initial data distributed by 
the log-correlated Gaussian measure $\mu$ in \eqref{gauss0}
(which is an invariant measure for the dynamics), 
if \eqref{K1} holds, 

\smallskip

\item 
does not converge to any meaningful limit if~\eqref{K2} holds.

\end{itemize}

\noi
See 
\cite[Subsection 1.2]{OST}
for other dynamical models related to the log-correlated Gibbs measures.

\end{remark}

\section{Preliminary lemmas}
\subsection{Deterministic estimates}

We first recall Young's inequality in the general setting;
see \cite[Theorem 156 on p.\,111]{HLP}.

\begin{lemma}\label{LEM:Young}
	Let $f$ be a strictly increasing function on $\R_+$
	such that   $f(0) = 0$ and its inverse $f^{-1}$ is also strictly increasing.
	Then, for any $a, b \ge 0$, we have
	\begin{align}
		ab \le \int_0^a f(x) dx   +  \int_0^b f^{-1}(x) dx 
		\label{Young1}
	\end{align}
	
	\noi
	with equality if and only if $b = f(a)$.
	In particular, applying \eqref{Young1}
	to  $f(x) = e^x - 1$ and $f^{-1}(x) = \log(1+x)$
	\textup{(}with $b$ replaced by $b-1$\textup{)}, 
	we have 
	\begin{align}
		ab \le e^{a} + b \log b - b
		\label{Young2}
	\end{align}
	
	\noi
	for any $a \ge  0$ and $b \ge 1$.
	
\end{lemma}

\subsection{Tools from stochastic analysis}

In this subsection, we state several useful lemmas
from stochastic analysis.
We first state  the Wiener chaos estimate
(\cite[Theorem~I.22]{Simon});
see \cite[Lemma 3.2]{OTz1}
for the following particular version.

\begin{lemma}\label{LEM:hyp}
Let $\left\{g_n\right\}_{n\in \Z^d}$ be a sequence of independent standard real-valued Gaussian
random variables. Given $k\in \N$, let $\left\{P_j\right\}_{j\in \N}$ be a sequence of polynomials in 
$\bar{g}=\left\{g_n\right\}_{n\in \Z^d}$ of degree at most $k$. Then, for any finite $p\ge 1$, we have
\begin{align*}
\bigg\|\sum_{j\in \N}P_j(\bar{g})\bigg\|_{L^p(\Omega)}
\le (p-1)^{\frac{k}{2}}\bigg\|\sum_{j\in \N}P_j(\bar{g})\bigg\|_{L^2(\Omega)}.
\end{align*}
\end{lemma}

Next, we recall the following orthogonality result
\cite[Lemma 1.1.1]{Nua}.

\begin{lemma}\label{LEM:ort}
Let $f$ and $g$ be jointly Gaussian random variables with mean zero 
and variances $\s_f$
and $\s_g$.
Then, we have 
\begin{align*}
\E\big[ H_k(f; \s_f) H_\l(g; \s_g)\big] = \dl_{k\l} k! \big\{\E[ f g] \big\}^k, 
\end{align*}

\noi
where
 $H_k (x,\s)$ denotes the Hermite polynomial of degree $k$ with a variance parameter $\s$.
\end{lemma}

The following convergence result
follows from a standard computation, using Lemmas~\ref{LEM:hyp}
and~\ref{LEM:ort};
see, for example, \cite{OTh} for the proof
when $d = 2$, 
which can be easily generalized to any dimension $d \in \N$.

\begin{lemma}\label{LEM:conv0}
Let $k \in \N$.
Then, 
given any finite $p \ge 1$, 
$\int_{\T^d} :\!(\pi_N u)^k\!:dx$
converges to a unique limit, 
denoted by 
$\int_{\T^d} :\!u^k\!:dx$, 
in $L^p(\mu)$ and $\mu$-almost surely, as $N \to \infty$, 
where $\mu$ is as in~\eqref{gauss0}.
In particular, 
given a sequence  $\{\lambda_N\}_{N\in \N}$  of positive numbers
tending to $0$ as $N \to \infty$, 
$\ld_N R_N(u)$ converges to $0$
in $L^p(\mu)$ and $\mu$-almost surely, as $N \to \infty$, 
where
 $R_N(u)$ is as in  \eqref{Wick0}.

\end{lemma}

The next lemma plays an important role
in studying convergence
of the indicator function $\ind_{\{|\int_{\T^d} \, :u_N^2: \, dx| \, \leq K_N\}} $;
see   
\cite[Lemma 2.4]{OST}
for the proof.
See also \cite[Remark 5.12]{CLO}.

\begin{lemma}\label{LEM:BM}
Let $\mu$ be as in \eqref{gauss0}.
Then, we have 
\begin{align*}
\mu \bigg(\int_{\T^d}:\!  u^2    \!: dx = K\bigg) = 0
\end{align*}

\noi
for any $K \in \R$, 
where $\int_{\T^d} :\!u^2\!: dx$ is the limit of 
$\int_{\T^d} :\!(\pi_N u)^2\!:dx$
as $N \to \infty$.

\end{lemma}

\subsection{Variational formulation}

We prove Theorem \ref{THM:1}, 
using a variational formula 
for the partition function, recently popularized 
in a  seminal work \cite{BG}
by Barashkov and Gubinelli;
 see also \cite{OOT, OST, OOT2, TW}.
First, we introduce some notations. 
Fix a  probability space $(\O, \F, \mathbb P)$.
Let $W(t)$ be a cylindrical Brownian motion in $L^2(\T^d)$.
Namely, we have
\begin{align*}
W(t) = \sum_{n \in \Z^d} B_n(t) e_n,
\end{align*}

\noi
where  
$\{B_n\}_{n \in \Z^d}$ is a sequence of mutually independent complex-valued\footnote{By convention, we normalize $B_n$ such that $\text{Var}(B_n(t)) = t$. In particular, $B_0$ is  a standard real-valued Brownian motion.} Brownian motions conditioned  that 
$\cj{B_n}= B_{-n}$, $n \in \Z^d$. 
Then, we define a centered Gaussian process $Y(t)$
by 
\begin{align}
Y(t)
=  \jb{\nabla}^{-\frac d2}W(t).
\label{P2}
\end{align}

\noi
In the following, we use the shorthand notation:  $Y = Y(1)$.
Then, 
we have $\Law(Y) = \mu$, 
where $\mu$ is the log-correlated Gaussian measure defined in~\eqref{gauss0}.
Given $N \in \N$, 
we set   $Y_N = \pi_NY $.
such that 
 $\Law(Y_N) = (\pi_N)_*\mu$, 
i.e.~the pushforward of $\mu$ under 
the frequency projector $\pi_N$ in~\eqref{pi1}.

Next, let $\Ha$ denote the space of drifts, 
which are progressively measurable processes 
belonging to 
$L^2([0,1]; L^2(\T^d))$, $\PP$-almost surely. 
Then,  the  Bou\'e-Dupuis variational formula \cite{BD, Ust}
reads as follow; see \cite{Zhang} and \cite[Appendix A]{TW}
for infinite-dimensional versions.

\begin{lemma}\label{LEM:var}
Given $N \in \N$, let $Y_N$ be as above.
Suppose that  $F:C^\infty(\T^d) \to \R$
is measurable such that $\E\big[|F(Y_N)|^p\big] < \infty$
and $\E\big[|e^{F(Y_N)}|^q \big] < \infty$ for some $1 < p, q < \infty$ with $\frac 1p + \frac 1q = 1$.
Then, we have
\begin{align}
 \log \E\Big[e^{F(Y_N)}\Big]
= \sup_{\dr \in \mathbb H_a}
\E\bigg[ F( Y_N + \Dr_N) - \frac{1}{2} \int_0^1 \| \dr(t) \|_{L^2_x}^2 dt \bigg], 
\label{var1}
\end{align}

\noi
where  
 the expectation $\E = \E_\PP$
is taken  with respect to the underlying probability measure~$\PP$.
Here, 
$\Dr_N = \pi_N \Dr$, 
where the process $\Dr$  is  defined by 
\begin{align}
\Dr (t)  =  \int_0^t \jb{\nabla}^{-\frac d2} \dr(t') dt'.
\label{Dr1}
\end{align}

\noi
\end{lemma}

We conclude this section by stating basic lemmas
in applying the variational formula (Lemma~\ref{LEM:var});
see \cite[Lemmas 3.2 and  3.5]{OST}.

\begin{lemma}\label{LEM:Dr}
\textup{(i)} 
Let $\eps > 0$. Then, given any finite $p \ge 1$, 
we have 
\begin{align*}
\E 
\Big[  \|Y_N\|_{W^{-\eps,\infty}}^p
+ \|:\!Y_N^2\!:\|_{W^{-\eps,\infty}}^p
+ \| :\! Y_N^3 \!:  \|_{W^{-\eps,\infty}}^p
\Big]
\leq C_{\eps, p} <\infty,
\end{align*}

\noi
uniformly in $N \in \N$.
	
\smallskip
	
\noi
\textup{(ii)} For any $\dr \in \Ha$, we have
\begin{align*}
		\| \Dr \|_{H^{\frac d2}}^2 \leq \int_0^1 \| \dr(t) \|_{L^2}^2dt.
	\end{align*}
\end{lemma}

\begin{lemma}\label{LEM:Dr1}
Given $N \in \N$, 
let $Y_N = \pi_N Y(1)$, 
where $Y$ is as in \eqref{P2}.
Then,  there exist small $\eps > 0$ and  a constant $c_0=c_0(\eps)>0$ 
	such that for any $\delta>0$, we have
\begin{align*}
		\bigg|\int_{\T^d}:\! Y_N^3\!:\Dr dx\bigg|
		&\le
		C(\delta)\|:\! Y_N^3\!:\|_{W^{-\eps,\infty}}^2
		+\delta\|\Dr\|_{H^{\frac d2}}^2,\\
		\bigg|\int_{\T^d}:\! Y_N^2\!:\Dr^2 dx\bigg|
		&\le
		C(\delta)\|:\! Y_N^2\!:\|_{W^{-\eps,\infty}}^4
		+\delta\Big(\|\Dr\|_{H^{\frac d2}}^2
		+\|\Dr\|_{L^4}^4\Big),\\
		\bigg|\int_{\T^d}Y_N\Dr^3 dx\bigg|
		&\le
		C(\delta)\|Y_N\|_{W^{-\eps,\infty}}^{c_0}
		+\delta
		\Big(\|\Dr\|^2_{H^{\frac d2}}
		+\|\Dr\|_{L^4}^4\Big),
	\end{align*}

	\noi
	uniformly in $N\in\N$.
\end{lemma}

\section{Weakly coupling regime: normalizability}\label{SEC:nor}

In this section, we present a proof of 
Theorem \ref{THM:1}\,(i). 
Let $K = \lim_{N \to \infty}K_N \in (0, \infty]$.
Then, it follows from 
Lemmas \ref{LEM:conv0}
and \ref{LEM:BM}
 that 
$\ind_{\{|\int_{\T^d} \, :(\pi_N u)^2: \, dx| \, \leq K_N\}}e^{\ld_N  R_N (u)}$
converges  
to 
$\ind_{\{|\int_{\T^d} \, :u^2: \, dx| \, \leq K\}}$
in   probability (with respect to  $\mu$) 
as $N \to \infty$.
Then, the desired $L^p(\mu)$-convergence~\eqref{exp2} follows
from the uniform integrability bound~\eqref{exp1};
see \cite[Remark 3.8]{Tz1}.
See also the discussion at the end of Section 2  in \cite{OTh}.

Given $N \in \N$, 
define
$ \W_N(\dr)$ by 
\begin{align*} 
    \W_N(\dr)=\E\bigg[
\ld_N    R_N(Y+\Dr )\cdot 
    \ind_{\{|\int_{\T^d} \, 
    :(Y_N+\Dr_N)^2: \, dx| \,
    \leq K_N\}}
    - \frac 12 \int_0^1 \| \dr(t)\|_{L^2_x} ^2 dt
    \bigg], 
\end{align*}

\noi
where $\Dr$ is as in \eqref{Dr1}.
Then, from 
Lemma \ref{LEM:var} with \eqref{P2}, 
we see that \eqref{exp1} follows once we prove 
\begin{align}\label{Y1}
\begin{split}
\sup_{N \in \N}    \log \E_{\mu}
    \bigg[\exp\Big(
\ld_N    R_N (u)\cdot\ind_{\{|\int_{\T^d} \, :(\pi_N u)^2: \, dx| \,
    \leq K_N\}}
    \Big)\bigg]
=\sup_{N \in \N} \sup_{\dr\in \mathbb H_a}
        \W_N(\dr) < \infty.
\end{split}
\end{align}

\noi
 From \eqref{Wick1} and  \eqref{Wick0}
 with \eqref{H1a} and the following identity
 (see \cite[(1.12)]{GKO}):
\begin{align*}
H_k(x+y; \s )
&  = 
\sum_{\l = 0}^k
\begin{pmatrix}
k \\ \l
\end{pmatrix}
 x^{k - \l} H_\l(y; \s), 
\end{align*}

\noi
we have 
\begin{align}
\begin{split}
\ld_N R_N(Y+\Dr)  & = 
		\ld_N\int_{\T^d}  :\! Y_N^4 \!:  dx
		+4\ld_N \int_{\T^d}  :\! Y_N^3 \!:  \Dr_N dx
		+6\ld_N \int_{\T^d}  :\! Y_N^2 \!:  \Dr_N^2 dx
		\\
		&\hphantom{X}
		+ 4\ld_N\int_{\T^d} Y_N  \Dr_N^3 dx
		+ \ld_N \int_{\T^d} \Dr_N^4 dx.
	\end{split}
	\label{Y0bis}
\end{align}	

\noi
By applying  Lemma \ref{LEM:Dr1} and  Lemma \ref{LEM:Dr}
to  \eqref{Y0bis}, we have 
\begin{align}
 \W_N(\dr)& \le 
 C_0 + \E\bigg[
2 \ld_N \| \Dr_N\|_{L^4}^4
\cdot\ind_{\{|\int_{\T^d} \, : (Y^2_N +\Dr_N)^2: \, dx| 
\leq K\}} -  \frac 1{4} \int_0^1 \| \dr(t)\|_{L^2_x} ^2 dt\bigg], 
\label{Y2}
\end{align}

\noi
uniformly in $N \in \N$ and $\ta \in \Ha$.	
Hence, 	
\eqref{Y1}
(and thus \eqref{exp1}) follows
from  \eqref{Y2}
once we prove the following proposition.

\begin{proposition}\label{PROP:H1}
	
There exists small $\ld_* > 0$ 
such that if
\eqref{K1} 
holds, 
then we have 
	\begin{align}
		\sup_{\dr\in \mathbb H_a}
		\E\bigg[
		\ld_N \| \Dr_N\|_{L^4}^4
		\cdot\ind_{\{|\int_{\T^d} \, : (Y^2_N +\Dr_N)^2: \, dx| 
			\leq K_N\}}
		-  \frac 1{10} \int_0^1 \| \dr(t)\|_{L^2_x} ^2 dt\bigg]
		\les 1, 
		\label{H1}
	\end{align}
	
\noi
uniformly in $N \in \N$.
	
\end{proposition}

%

Suppose that $\ld_N$ decays at a rate of the form
 $\ld_N \les N^{-\kk}$
for some $\kk > 0$.
Then, by Sobolev's inequality, interpolation, and 
(standard) Young's inequality, we have 
\begin{align}
	\begin{split}
		\ld_N \| \Dr_N\|_{L^4}^4
		& \le C \| \Dr_N\|_{H^{\frac {d-\kk}4}}^4
		\le C' \| \Dr_N\|_{H^{\frac d2}}^\frac{2(d-\kk)}{d}
		\| \Dr_N\|_{L^2}^\frac{2(d+\kk)}{d}\\
		& \le \frac 1{10} \|\Dr_N\|_{H^\frac d2}^2 
		+ C'' \| \Dr_N\|_{L^2}^\frac{2(d+\kk)}{\kk}, 
	\end{split}
	\label{H2}
\end{align}

\noi
where the first term on the right-hand side
is then controlled by 
Lemma \ref{LEM:Dr}\,(ii).
When 
$\ld_N \sim (\log N)^{-1}$ (essentially corresponding to $\kk = 0$),
such an argument does not work, 
 exhibiting the critical nature
of our  problem.
While our argument is motivated by 
those in 
\cite[Subsection~5.6]{OOT}
and \cite[Subsection~3.2]{OOT2}, 
there is an additional difficulty
in our current problem as compared to 
those in \cite{OOT, OOT2}
in the following sense.
In \cite{OOT, OOT2}, the essential part
in estimating 
the potential energies
was reduced
to estimating $\|\Dr_N\|_{L^2}^6$; 
see \cite[(5.76)]{OOT}
and \cite[(3.21)]{OOT2}.
On the other hand, 
when $\kk = 0$, 
\eqref{H2} would give us
an infinite power of 
the $L^2$-norm of $\Dr_N$;
see the first term on the right-hand side of \eqref{H6}.

\begin{proof}[Proof of Proposition \ref{PROP:H1}]

	On $A_N := 
	\big\{|\int_{\T^d} \, : \!(Y^2_N +\Dr_N)^2\!: \, dx| 
	\leq K_N\big\}$, we have 
	\begin{align}
\|\Dr_N\|_{L^2}^2  \le K_N +\s_N  + 2 \bigg|\int_{\T^d} Y_N \Dr_N dx\bigg|, 
\label{H3}
\end{align}

	\noi
	where $\s_N \sim \log N$ is as in \eqref{sigma1}.
	First, 
	suppose that  we have
	\begin{align}
		\| \Dr_N  \|_{L^2}^2 \les K_N + \s_N. 
		\label{H4}
	\end{align}
	
	\noi
	Then, 
	from  Sobolev's inequality, interpolation,  \eqref{H3}, 
	and  \eqref{H4}
	with \eqref{sigma1} and \eqref{K1}
	followed by 
	Lemma \ref{LEM:Dr}\,(ii), we obtain
	\begin{align*}
			&  \ld_N \| \Dr_N\|_{L^4}^4
\cdot\ind_{\{|\int_{\T^d} \, : (Y^2_N +\Dr_N)^2: \, dx| 
				\leq K_N\}}
			\\ 
			& \quad  
\le C \ld_N \| \Dr_N\|_{L^2}^2 \| \Dr_N\|_{H^{\frac d2}}^2
\cdot\ind_{\{|\int_{\T^d} \, : (Y^2_N +\Dr_N)^2: \, dx| 	\leq K_N\}}\\
& \quad 	\le  C' \ld_*  \|  \Dr_N\|_{H^\frac d2 }^2
			\le  \frac{1}{10} \int_0^1 \| \dr(t) \|_{L^2_x}^2 dt , 
	\end{align*}
	
	\noi
	provided that $\ld_*$ is  sufficiently small.
	This yields \eqref{H1}
under the condition \eqref{H4}.

	In view of \eqref{H3}, it remains to consider  the case:
	\begin{align}
		\| \Dr_N  \|_{L^2}^2 \les \bigg| \int_{\T } Y_N  \Dr_N  dx \bigg|.
		\label{H5}
	\end{align}

\noi	
We note that the following argument holds
under a weaker assumption:	
	\begin{align}
		\ld_N \le \ld_* (\log N)^{-1}, \quad N \in \N.
		\label{H1b}
	\end{align}
	By applying  \eqref{Young2} in Lemma \ref{LEM:Young}, \eqref{H1b}, 
	and 
Bernstein's inequality (recall that 
	$\supp \ft \Dr_N \subset \{|n|\le N\}$), we have  
	\begin{align}
		\begin{split}
			\ld_N \| \Dr_N\|_{L^4}^4
			& \les  \ld_*^\frac{1}{2} (\log N)^{-1} \| \Dr_N\|_{L^2}^2 
			\Bigg(\ld_*^\frac{1}{2} \| \Dr_N\|_{L^2}^2 \frac{\| \Dr_N\|_{H^{\frac d2}}^2}{\| \Dr_N\|_{L^2}^2}
			\Bigg)\\
			& \le  \ld_*^\frac{1}{2} (\log N)^{-1} \| \Dr_N\|_{L^2}^2  
			e^{\ld_*^\frac{1}{2} \|\Dr_N\|_{L^2}^2}\\
			& \quad +  \ld_*^\frac{1}{2} (\log N)^{-1}
			\| \Dr_N\|_{H^{\frac d2}}^2
			\log \frac{\| \Dr_N\|_{H^{\frac d2}}^2}{\| \Dr_N\|_{L^2}^2}\\
			& \les  \ld_*^\frac{1}{2} (\log N)^{-1}   e^{2\ld_*^\frac{1}{2}  \|\Dr_N\|_{L^2}^2}
			+   \ld_*^\frac{1}{2} 
			\| \Dr_N\|_{H^{\frac d2}}^2.
		\end{split}
		\label{H6}
	\end{align}
	
	\noi
	We now claim that there exists
	a non-negative random variable $X_N(\o)$
	with 
	\begin{align}
		\sup_{N \in \N} \E[ X_N] < \infty
		\label{H6a}
	\end{align}
	such that 
	\begin{align}
		e^{ 2\ld_*^\frac{1}{2}  \|\Dr_N\|_{L^2}^2}
		\les   1+   \| \Dr_N  \|_{H^\frac d2 }^2
		+ X_N(\o).
		\label{H7}
	\end{align}
	
	\noi
	Then, 
	\eqref{H1} follows from 
	\eqref{H6}, \eqref{H7}, and 
	Lemma \ref{LEM:Dr}\,(ii), 
	provided that $\ld_*$ is sufficiently small.
	
	\medskip
	
	The remaining part of the proof
	is devoted to proving \eqref{H7}.
	We proceed as in the proof of 
	\cite[Lemma 3.6]{OOT2}
	(see also \cite[Subsection~5.6]{OOT}).
	Define the  sharp frequency projections $\{\Pi_j\}_{j \in \N}$
	by setting $\Pi_1 = \pi_2$
	and $\Pi_j = \pi_{2^j} - \pi_{2^{j-1}}$.
	We also set $\Pi_{\le j} = \sum_{k = 1}^j \Pi_k$
	and $\Pi_{> j} = \Id - \Pi_{\le j}$.
	Then,
	write  $\Dr_N $ as 
	\begin{align*}
		\Dr_N   = \sum_{j=1}^\infty (\al_j \proj_j Y_N  + w_j),
	\end{align*}
	where
	\begin{align*}
		\al_j &:=
		\begin{cases}
			\frac{\jb{\Dr_N , \proj_j Y_N }}{\|\proj_j Y_N \|_{L^2}^2},  & \text{if } \| \proj_j Y_N  \|_{L^2} \neq 0, \\
			0, & \text{otherwise},
		\end{cases}
		\qquad
		\text{and}
		\qquad  
		w_j :=
		\proj_j \Dr_N  - \al_j \proj_j Y_N .
	\end{align*}
	
	\noi
	Noting that $w_j$ is orthogonal to $\proj_j Y_N $ and $Y_N $ in $L^2(\T^d)$, 
	we have 
	\begin{align}
		\| \Dr_N  \|_{L^2}^2
		&= \sum_{j=1}^\infty \Big( \al_j^2 \| \proj_j Y_N  \|_{L^2}^2 + \| w_j \|_{L^2}^2 \Big), \notag \\
		\int_{\T^d } Y_N  \Dr_N  dx
		&= \sum_{j=1}^\infty \al_j \| \proj_j Y_N  \|_{L^2}^2.
		\label{HH2}
	\end{align}

	Given small $N \in \N$, 
	fix  a random number $j_0 \in \N$ (to be chosen later).
	Then, arguing as in the proof of Lemma~3.6 in \cite{OOT2}
	(see \cite[(3.49)-(3.51)]{OOT2}), we obtain
	\begin{align}
		\begin{split}
			\bigg| \sum_{j=1}^\infty \al_j \| \proj_j Y_N  \|_{L^2}^2 \bigg|
			\les
			\| \Dr_N  \|_{H^\frac d2}
			\| \Pi_{> j_0} Y_N \|_{H^{-\frac d2}}
			+
			\| \proj_{\le j_0} Y_N  \|_{L^2}^2.
		\end{split}
		\label{HH3}
	\end{align}

	\noi
	Since $Y_N $ is spatially homogeneous, we have
	\begin{align}
		\|\Pi_{> j_0} Y_N  \|_{H^{-\frac d2}}^2
		&=  \int_{\T^d} :\! ( \jb{\nb}^{-\frac d2}\Pi_{> j_0} Y_N  )^2 \!: dx
		+  \E \big[( \jb{\nb}^{-\frac d2}\Pi_{> j_0} Y_N  )^2 \big], 
		\label{HH4}
	\end{align}
	
	\noi
	where the last term is independent of $x \in \T^d$
	(and hence we suppressed its $x$-dependence).
	From~\eqref{P2}, we have
	\begin{align}
		\wt \s_{j_0} :=   \E \big[( \jb{\nb}^{-\frac d2}\Pi_{> j_0} Y_N  )^2 \big]
		= \sum_{|n| > 2^{j_0}} \frac{1}{\jb{n}^{2d}} 
		\sim 2^{-dj_0}.
		\label{HH5}
	\end{align}

	\noi
	Proceeding as in the proof of Lemma 2.5 in \cite{OTh}
	with Lemma \ref{LEM:ort}, we have
	\begin{align}
		\begin{split}
			\E& \bigg[ 
			\Big( \int_{\T^d} :\! ( \jb{\nb}^{-\frac d2}\Pi_{> j_0} Y_N  )^2 \!: dx \Big)^2 \bigg]\\
			%
			& = \int_{\T^d_x \times \T^d_y}
			\E\Big[  H_2(\jb{\nb}^{-\frac d2}\Pi_{> j_0} Y_N (x); \wt \s_{j_0})
			H_2( \jb{\nb}^{-\frac d2}\Pi_{> j_0} Y_N (y); \wt \s_{j_0})\Big] dx dy\\
			& = 2\int_{\T^d_x \times \T^d_y}
			\Big\{ \E\big[  \jb{\nb}^{-\frac d2}\Pi_{> j_0} Y_N (x)
			\cdot \jb{\nb}^{-\frac d2}\Pi_{> j_0} Y_N (y)\big]\Big\}^2 dx dy\\
			& =  2
			\sum_{\substack{n_1, n_2\in \Z^d\\  2^{j_0} <  |n_j|\le N}}\frac{1}{\jb{n_1}^{2d}\jb{n_2}^{2d}}
			\int_{\T^3_x \times \T^3_y}
			e_{n_1 + n_2}(x-y) dx dy\\
			& 
			\les   \sum_{|n| > 2^{j_0}}\frac{1}{\jb{n}^{4d}}
			\sim 2^{-3dj_0}.
		\end{split}
		\label{HH6}
	\end{align}

	\noi
	Now, define a non-negative random variable $B_{1, N}(\o)$ by setting
	\begin{align}
		\begin{split}
			B_{1, N}(\o) & = 
			\bigg(\sum_{k = 1}^\infty
			2^{\frac 52 d k} \Big(  \int_{\T^d } :\! (\jb{\nb}^{-\frac d2} \proj_{> k}Y_N )^2 \!: dx\Big)^2 \bigg)^\frac{1}{2}.
		\end{split}
		\label{HH7}
	\end{align}

	\noi
	From \eqref{HH4}, \eqref{HH5},  and \eqref{HH7}, we obtain
	\begin{align}
		\| \Pi_{> j_0} Y_N \|_{H^{-\frac d2}}^2
		\les  2^{-\frac {5d}4j_0}B_{1, N}(\o)+ 
		2^{-dj_0}.
		\label{HH8}
	\end{align}

	\noi
	Let us now consider the second term on the right-hand side
	of \eqref{HH3}. As in \eqref{HH4}, write 
	\begin{align}
		\| \proj_{\le j_0} Y_N  \|_{L^2}^2
		=  \int_{\T^d} :\! (\proj_{\le j_0} Y_N )^2 \!: dx
		+  \E \big[ (\proj_{\le j_0}  Y_N )^2 \big] .
		\label{HH9}
	\end{align}
	
	\noi
	We have 
	\begin{align}
		\E \big[ (\proj_{\le j_0}  Y_N )^2 \big] 
		= \sum_{|n| \le  2^{j_0}}  \frac{1}{\jb{n}^{d}}
		\sim j_0.
		\label{HH10}
	\end{align}
	
	\noi
	Proceeding as in \eqref{HH6}, we have 
	\begin{align}
		&\E \bigg[ 
		\Big( \int_{\T^d} :\! ( \Pi_{k} Y_N  )^2 \!: dx \Big)^2 \bigg]
		\les  2^{-dk}.
		\label{HH11}
	\end{align}
	
	\noi
	As in \eqref{HH7}, 
	define a non-negative random variable $B_{2, N}(\o)$ by setting
	\begin{align}
		\begin{split}
			B_{2, N}(\o) & = 
			\sum_{k = 1}^\infty
			\bigg|  \int_{\T^d } :\! ( \proj_{ k}Y_N )^2 \!: dx\bigg|.
		\end{split}
		\label{HH12}
	\end{align}
	
	\noi
	Thus, from \eqref{HH9}, \eqref{HH10}, and \eqref{HH11}, 
	we have 
	\begin{align}
			\| \proj_{\le j_0} Y_N  \|_{L^2}^2
			&\les 
			B_{2, N}(\o) + j_0.
		\label{HH13}
	\end{align}

	We now choose $j_0 = j_0(\o)$ by 
	\begin{align}
		2^{\frac d2 j_0(\o)} \sim 2 + \| \Dr_N (\o) \|_{H^\frac d2}.
		\label{HH14}
	\end{align}

	\noi
	Then, 
	putting 
	\eqref{H5},  \eqref{HH2}, 
	\eqref{HH3},   \eqref{HH8},  and \eqref{HH13} together
	with \eqref{HH14}, we have 
	\begin{align*}
			\| \Dr_N  \|_{L^2}^2
			&    \leq
			\Big( 2^{-\frac  d2 j_0} +  2^{-\frac {5d}{8} j_0} 
			B_{1, N}^\frac{1}{2}(\o)\Big) \| \Dr_N  \|_{H^\frac d2} +  
			B_{2, N}(\o)
			+ j_0\\
			&    \les
			\log\big( 2+  \| \Dr_N  \|_{H^\frac d2 }\big) +  
			B_{1, N}^\frac 12 (\o)
			+ B_{2, N}(\o).
	\end{align*}
	
	\noi
	Hence, by Young's inequality
	and choosing $\ld_*>0$ sufficiently small, we obtain
	\begin{align*}
		e^{2 \ld_*^\frac{1}{2}  \|\Dr_N\|_{L^2}^2}
		\les   1+  \| \Dr_N  \|_{H^\frac d2 }^2
		+ X_N, 
	\end{align*}

	\noi
	yielding \eqref{H7}, 
	where $X_N = X_N(\o)$ is defined by 
	\begin{align}
		X_N(\o) = e^{\g  B_{1, N}^\frac 12 (\o)
			+ \g B_{2, N}(\o)}
		\label{HH16a}
	\end{align}
	
	\noi
	with some small constant $\g > 0$, 
	which we now choose to guarantee 
	the uniform bound \eqref{H6a}.
	From \eqref{HH7}, Minkowski's integral inequality, the Wiener chaos estimate (Lemma \ref{LEM:hyp}), 
	and~\eqref{HH6} (with $j_0$ replaced by $k$), 
	we see that 
	\begin{align*}
		\E \big[B_{1, N}^p\big] \les p 
	\end{align*}
	for any finite $p \geq 1$, uniformly in $N \in \N$.
	Then, it follows from \cite[Lemma 4.5]{TzBO} that there exists $\g > 0$ such that 
	\begin{align}
		\E \big[e^{\g B_{1, N}^\frac 12 }\big] \les 1, 
		\label{HH17}
	\end{align}
	
	\noi
	uniformly in $N \in \N$.
	A similar computation applied to \eqref{HH12} yields
	\begin{align}
		\E \big[e^{\g B_{2, N}}\big] \les 1, 
		\label{HH18}
	\end{align}

	\noi
	uniformly in $N \in \N$.
	Hence, the bound \eqref{H6a} follows
	from \eqref{HH16a}, \eqref{HH17}, and \eqref{HH18}.
\end{proof}

\section{Strongly  coupling regime: non-normalizability}
\label{SEC:nonnor}

In this section, 
we present a proof of Theorem \ref{THM:1}\,(ii)
 by following closely
the argument in \cite[Section 3]{OST}, 
which 
was
in turn  is inspired by  the recent works 
by the fourth author with Weber~\cite{TW}
and by the second and fourth authors with Okamoto
\cite{OOT, OOT2}.
For this purpose, 
we first recall notations and  preliminary results from  \cite{OST}.

Let $f: \R^d \to \R$ be a real-valued Schwartz function
with $\|f\|_{L^2(\R^d, \frac{dx}{(2\pi)^d})} = 1$
such that 
its Fourier transform $\ft f$ is 
supported  on $\{\xi \in \R^d:  |\xi| \le 1 \}$ with $\ft f(0) = 0$.
Define a function $f_M$  on $\T^d$ by 
\begin{align}
	f_M =  M^{-\frac d2} \sum_{\substack{|n| \le M}} \ft f\Big( \frac nM \Big) e_n, 
	\label{fMdef} 
\end{align}

\noi
where $\ft f = \F_{\R^d}(f)$ denotes the Fourier transform on $\R^d$.

\begin{lemma}[Lemma 3.3 in \cite{OST}]
\label{LEM:leo1}
	Let   $\al > 0$. Then,  we have
	\begin{align}
		\int_{\T^d} f_M^2 dx &= 1 + O( M^{-\al}), \label{fM0} \\
 \int_{\T^d}  f_M^4  dx &
 \sim M^d,  \label{fM1}\\
\int_{\T^d} (\jb{\nabla}^{-\frac d2} f_M)^2 dx 
& \sim M^{-d}
\label{fm2} 
\end{align}
	
\noi
for any $M \gg 1$.
	
\end{lemma}

We also  define $Q(u$) by 
\begin{align}
Q(u) =  \int_{\T^d}  u^4 dx.
\label{Q1}
\end{align}

\noi
As in \cite{OST}, 
the divergence of $Q(f_M)$
(see \eqref{fM1}) is what allows us to prove \eqref{Gibbs3}.

In the following, we split
the proof of Theorem \ref{THM:1}\,(ii) 
into two cases, 
depending on 
whether $\sup_{N\in \N}(\log N)^{-1} K_N<\infty$ or not.
In the former case, 
 we use exactly the
same drift $\dr^0$ as in~\cite{OST}, 
which we recall now.
%
As seen in \cite{OOT, OST, OOT2}, 
the main idea is to choose a drift $\dr_N$
in applying 
Lemma~\ref{LEM:var}
(with $F = \ld_N R_N$)
such that 
$\Dr_N
= \int_0^t \jb{\nabla}^{-\frac d2} \dr_N(t') dt'$  is of the form: 
\begin{align}
\text{``}\,\Dr_N = - Y_N + \text{a deterministic perturbation } h_N\,\text{''}
\label{Q1a}
\end{align}

\noi
such that $h_N$ has a bounded $L^2$-norm
but has a large $L^4$-norm
(see \eqref{fM0} and \eqref{fM1} above)
which dominates the last term in \eqref{var1}, yielding 
the desired divergence.
The issue is that $-Y_N = - \pi_N Y$ 
defined in \eqref{P2}
is a Brownian motion in time and thus is not differentiable in time.
Hence, we need to introduce a suitable   approximation 
to $Y_N$ in \eqref{Q1a}.

In \cite[Lemma 3.4]{OST}, 
given $M \gg 1$, we constructed 
the approximation process 
 $\z_M$ to $Y$ in~\eqref{P2}
by solving 
the stochastic  differential equation \cite[(3.18)]{OST}
on low frequencies  $\{|n| \leq M\}$
and 
setting 
 $\ft \z_{M}(n, t)  \equiv 0$ on high frequencies  $\{|n| > M\}$.
In \cite[(3.27)]{OST}, 
we then defined a drift $\dr^0$ by setting
\begin{align*}
	\dr^0 (t) 
	& = \jb{\nb}^{\frac d2} \bigg( -\frac{d}{dt} \z_M(t) + \sqrt{ \al_{M, N}} f_M \bigg), 
\end{align*}

\noi
where $\al_{M, N}$ is as in \cite[(3.25)]{OST}.
We note from \cite[(3.23)]{OST} that 
$\jb{\nb}^{\frac d2} \frac{d}{dt} \z_M
\in L^2([0,1]; L^2(\T^d))$, $\PP$-almost surely, 
and thus 
$ \dr^0\in\mathbb H_a$.
As in \eqref{Dr1}, 
we then set 
\begin{align*}
\Dr^0 
= \int_0^1 \jb{\nb}^{-\frac d2} \dr^0(t) \, dt = - \z_M + \sqrt{ \al_{M, N}} f_M.
\end{align*}

\noi
From \cite[(3.26) and (3.34)]{OST}
and the frequency supports of $f_M$ and $\z_M$, 
 we have 
\begin{align}\label{Q4}
\alpha_{M,N}& =\s_M(1+o(1))\sim \log M,\\
\E\big[\| \Dr^{0}\|^2_{H^{\frac{d}{2}}}\big]
& \le \E\bigg[\int_0^1 \| \dr^{0}(t)\|_{L^2_x} ^2 dt\bigg]
\les M^{d} \log M ,
\label{Q6}\\
\pi_N \Dr^{0}& =\Dr^{0}_N=\Dr^{0}
\notag 
\end{align}

\noi
for any $N\ge M\gg 1$.
On the other hand, from 
\eqref{Q1}, 
\eqref{fM1},  \eqref{Q4}, 
and \eqref{K2}, 
we have
\begin{align}
\ld_N Q(\Dr^0)
\sim \ld_N   M^d
(\log M)^2.
\label{Q7}
\end{align}

\noi
By setting $M = N$, 
it follows from \eqref{Q6} and \eqref{Q7} with \eqref{K2}
that 
$\ld_N Q(\Dr^0)$ dominates the last term 
in \eqref{var1}
which is the  main obstruction
in achieving the desired divergence.

%
%
%

The following lemma follows from a slight modification of the proof of \cite[(3.42)]{OST}.

\begin{lemma}\label{LEM:pa5}
Let $\{K_N\}_{N\in \N}$ be a non-decreasing sequence
of positive numbers.
There exists  $M_0=M_0(K_1) \in \N$ such that
	\begin{align*}
		\PP\bigg( \Big|\int_{\T^d} (:{Y_N^2}: + 2 Y_N \Dr^{0} + (\Dr^{0})^2) dx \Big| \le K_N \bigg) 
		\ge \frac 12
	\end{align*}
	\noi
for any $N\ge M\ge M_0(K_1)$.
\end{lemma}

\medskip

We now present a proof of 
Theorem \ref{THM:1}\,(ii).

\begin{proof}[Proof of Theorem \ref{THM:1}\,(ii)]

We only prove the divergence in \eqref{Gibbs3}, 
since the claimed non-convergence then follows from Lemma \ref{LEM:conv0}.
Noting that
\begin{align*}
& \E_\mu\Big[\exp\big(\min (  \ld_N R_N(u),L) \big)
\cdot \ind_{\{ |\int_{\T^d} \, : (\pi_N u)^2 :\, dx | \le K_N\}}   \Big] \\
& \quad \ge 
\E_\mu\Big[\exp\big(\min ( \ld_N  R_N(u),L)
\cdot \ind_{\{ |\int_{\T^d} \, : (\pi_N u)^2 :\, dx | \le K_N\}}\big)   \Big] -1,  
\end{align*}

\noi
it suffices to prove 
\begin{align}
\liminf_{N\to \infty} \lim_{L\to \infty}  \E_\mu\Big[\exp\big(\min{(  \ld_N  R_N(u),L)} 
\cdot \ind_{\{ |\int_{\T^d} \, : (\pi_N u)^2 :\, dx | \le K_N\}}\big)   \Big] =  \infty.
\label{pa0}
\end{align}

\noi
We split the proof into two cases.

\medskip

\noi
$\bullet$ {\bf Case 1:}
$\sup_{N\in \N}(\log N)^{-1} K_N<\infty$.

In this case, from \eqref{K2}, we have 
\begin{align}\label{A0}
	\ld_N\ge C\ld^{*}(\log N)^{-1}, \quad N \in \N,
\end{align}

\noi
for some $C>0$.
In the following, we take $N \ge M \gg 1$.

From \eqref{Y0bis} (with $\Dr_N$ replaced by $\Dr^0$) and Lemma \ref{LEM:Dr1},
we have 
\begin{align}
\begin{split}
\ld_N R_N(Y + \Dr^{0})
&\ge (1-\dl) \ld_N Q(\Dr^{0}) \\
			& \quad
			- c(\dl)\ld_N\Big(
			\|:\! Y_N^3 \!:\|_{W^{-\eps,\infty}}^2
			+\|:\! Y_N^2 \!:\|_{W^{-\eps,\infty}}^4 
			+\|Y_N\|_{W^{-\eps,\infty}}^{c_0}
			\Big) \\
			&\quad
			- \ld_N |R_N(Y)|
			-\dl \ld_N \|\Dr^0\|_{H^\frac d2}^2.
		\end{split}
		\label{A1}
	\end{align}
	
\noi
See \cite[(3.41)]{OST}.
Proceeding as in 
 \cite[(3.31)-(3.33)]{OST}
 with  \eqref{Q4},  Lemma \ref{LEM:leo1}
 (see also \eqref{Q7}),  and  Lemma \ref{LEM:pa5}, 
we have  
 \begin{align}
\begin{split}
\E &\Big[\min{((1-\dl) \ld_N Q(\Dr^{0}),L)}\cdot 
		\ind_{\{|\int_{\T^d} \, 
			:(Y_N+\Dr^{0})^2: \, dx| \,
			\leq K_N\}}\Big]\\
& \ge  C\lambda_N \alpha^{2}_{M,N} M^d
-C(\delta) \ld_N\alpha_{M,N}\E[\z^2_M]\|f_M\|^2_{L^2} \\
& \ge C'\lambda_N M^d (\log M)^2-C''(\delta)\ld_N(\log M)^2.
\end{split}
\label{A2}
\end{align}

\noi
for any small $\dl > 0$
and 
$N \ge M \gg 1$, 
provided that 
 $L\gg 
 \ld_N \alpha^{2}_{M,N}Q(f_M)\sim 
 \lambda_N \alpha^{2}_{M,N}M^d $.
Thus, it follows
from  
the variational formula 
(Lemma \ref{LEM:var}), 
\eqref{A1}, \eqref{A2}, 
Lemmas \ref{LEM:Dr} and \ref{LEM:conv0},
and 
\eqref{Q6}
that 
\begin{align}\label{A3}
\begin{split}
& \log  {\E_\mu \Big[\exp\big(\min{(  \ld_N R_N(u),L)} \cdot \ind_{\{ |\int_{\T^d} \, : u_N^2 :\, dx | \le K_N\}}\big)  
 \Big]}\\
& \quad \ge \E\bigg[
\min{(\ld_N  R_N(Y+\Dr^{0}),L)}\cdot 
\ind_{\{|\int_{\T^d} \, 
				:(Y_N+\Dr^{0})^2: \, dx| \,
				\leq K_N\}}
			-\frac 12 \int_0^1 \| \dr^{0}(t)\|_{L^2_x} ^2 dt
			\bigg]\\
& \quad \ges \lambda_N M^d (\log M)^2
-C(\delta)\ld_N(\log M)^2
- C' M^d \log M
-C''(\dl)
\end{split}
\end{align}

\noi
\noi
for any small $\dl > 0$
and 
$N \ge M \gg 1$, 
provided $L\gg \ld_N \alpha^{2}_{M,N} M^d.$
In view of \eqref{A0}, 
by setting $M = N$ in \eqref{A3}, 
 taking $L\to \infty$, 
and then $N \to \infty$, 
 \eqref{pa0} follows from \eqref{A3}, 
provided that $\ld^*$ is sufficiently large.

\smallskip

\noi
$\bullet$ {\bf Case 2:}
$\sup_{N\in \N}(\log N)^{-1} K_N = \infty$.

In this case, from \eqref{K2}, we have 
\begin{align}
	\ld_N\ge C\ld^{*}K_N^{-1}, \quad N \in \N,
\label{B1}
\end{align}
for some $C>0$.
In this case, we have  $K_N \gg \int_{\T^d} Y_N^2(\o) dx  \sim C(\o)  \log N$.
Namely, we can   take a drift $\ta$ 
such that $\Dr$  defined in \eqref{Dr1} 
has a much larger   $L^2$-norm than that of $Y_N$.
This,  in particular, allows
us to choose a {\it deterministic} drift
such that 
$\ld_N Q(\Dr_N) \gg  \|\Dr_N\|_{H^\frac d2}^2$
to drive the desired divergence;
see \eqref{B5} and \eqref{B5a}
(with $M = N$).
%
%

Given  $M\gg 1$,
let  $f_M$ be 
as in \eqref{fMdef} and 
we define a drift $\dr_\g$ by 
\begin{align}
	\dr_{\gamma} (t) 
	& =\sqrt{\gamma K_M}\, \jb{\nb}^{\frac d2} f_M,
\label{B2}
\end{align}

\noi
for some small $\g > 0$ is independent of $M$ (to be chosen later). 
We then set 
\begin{align}
\Dr_\gamma 
= \int_0^1 \jb{\nb}^{-\frac d2} \dr_\gamma(t) \, dt = \sqrt{\gamma K_M} f_M.
\label{B3}
\end{align}

\noi
From the frequency support of $f_M$, we have 
\begin{align*}
	\pi_N \Dr_\gamma = \Dr_{\gamma}
\end{align*}

\noi
for any  $N\ge M\gg  1$.
From  \eqref{B3}, Cauchy-Schwarz's inequality (in time), 
\eqref{B2}, Bernstein's inequality
(with $\supp \ft f_M \subset \{ |n|\le M\}$), 
and    Lemma \ref{LEM:leo1}, we have 
\begin{align}\label{B5}
\| \Dr_\gamma\|_{H^\frac d2}^2
\le
\int^{1}_{0}\|\dr_\gamma(t)\|^2_{L_x^2}dt 
\les \gamma K_M M^d
\end{align}

\noi
for any $M \gg1 $.
From  \eqref{Q1} and Lemma \ref{LEM:leo1}, we also have 
\begin{align}
Q(\Dr_\g) 
\sim \g^2 K_M^2 M^d
\label{B5a}
\end{align}

\noi
for any $M \gg1 $.

We claim that by choosing  $\g> 0$ sufficiently small, we have 
\begin{align}
\PP\bigg( \Big|\int_{\T^d} :{Y_N^2}: + 2 Y_N \Dr_\gamma + \Dr_\gamma^2 \, dx \Big| \le K_N \bigg)
	\ge \frac 12 
\label{B6}
\end{align}

\noi
for any $N =  M \gg 1$.
Then, 
from  
the variational formula 
(Lemma \ref{LEM:var}), 
\eqref{A1}
(with $\Dr^0$ replaced with $\Dr_\g$), 
\eqref{B5}, \eqref{B5a}, 
and Lemmas \ref{LEM:Dr} and \ref{LEM:conv0},
we have 
\begin{align}
\begin{split}
& \log  {\E_\mu \Big[\exp\big(\min{(  \ld_N R_N(u),L)} \cdot \ind_{\{ |\int_{\T^d} \, : u_N^2 :\, dx | \le K_N\}}\big) 
 \Big]}\\
& \quad   \ge \E\bigg[\min{(\ld_N  R_N(Y+\Dr_\g),L)}\cdot 
\ind_{\{|\int_{\T^d} \, 
				:(Y_N+\Dr_\g)^2: \, dx| \,
				\leq K_N\}}
			-\frac 12 \int_0^1 \| \dr_\g(t)\|_{L^2_x} ^2 dt
			\bigg]\\
&\quad \ges  \ld_N \g^2 K^2_M M^d
- C\g K_M M^d
- C'(\delta)
\end{split}
\label{B7}
\end{align}

\noi
for any small $\dl > 0$
and 
$N =  M \gg 1$, 
provided 
$L\gg \ld_N \g^2 K^2_M M^d$.
In view of \eqref{B1}, 
by setting $M = N$ in \eqref{B7}, 
 taking $L\to \infty$, 
and then $N \to \infty$, 
 \eqref{pa0} follows from \eqref{B7}, 
provided that $\ld^* = \ld^*(\g)$ is sufficiently large.

It remains to prove \eqref{B6}.
From Lemma \ref{LEM:Dr}, we have 
\begin{align}
\begin{split}
\E& \bigg[\Big|\int_{\T^d} :{Y_N^2}: + 2 Y_N \Dr_\gamma + (\Dr_\gamma)^2 dx\Big|^2\bigg] \\
& \les
1+  \E\bigg[\Big|\int_{\T^d}Y_N \Dr_\gamma\, dx \Big|^2\bigg]
+\E\bigg[\Big(\int_{\T^d}\Dr^2_\gamma \,dx\Big)^2\bigg]\\
\end{split}
\label{B8}
\end{align}

\noi
From 
 \eqref{B3}, \eqref{P2},  and \eqref{fm2} in Lemma \ref{LEM:leo1}, we have
\begin{align}\label{B9}
\begin{split}
\E\bigg[\Big|\int_{\T^d}Y_N \Dr_\gamma\, dx \Big|^2\bigg]
& =\g K_M \E\bigg[\Big|\int_{\T^d}Y_N f_M dx \Big|^2\bigg] =\g K_M \sum_{|n|\le M} \jb{n}^{-d}|\widehat{f}_M(n)|^2\\
& \les \g K_M M^{-d}.
\end{split}
\end{align}

\noi
From \eqref{B3} and  \eqref{fM0}, we have 
\begin{align}\label{B10}
\E\bigg[\Big(\int_{\T^d}\Dr^2_\gamma \,dx\Big)^2\bigg]
= \gamma^2 K^2_M\|f_M\|^4_{L^2}
\sim \gamma^2 K^2_M 
\end{align}

\noi
for any $M \gg 1$.
Hence, by applying  Chebyshev's inequality, \eqref{B8}, 
 \eqref{B9}, and  \eqref{B10}, we obtain
\begin{align}
\begin{split}
\PP & \bigg( \Big|\int_{\T^d} (:{Y_N^2}: + 2 Y_N \Dr_\gamma + (\Dr_\gamma)^2) dx \Big| \ge K_N \bigg)\\
& \les \frac{1}{K^2_N}+ \g \frac{K_M}{M^d K^2_N}+ \g ^2 \frac{K^2_M}{K^2_N}.
\end{split}
\label{B11}
\end{align}

\noi
Finally, 
recalling that $K_N \ges \log N \to \infty$ as $N \to \infty$, 
the bound  
\eqref{B6} follows from 
setting  $M=N$ sufficiently large and choosing $\g> 0$ sufficiently small
in \eqref{B11}.
\end{proof}

\begin{ackno} \rm
 The authors would like to thank an anonymous referee for the helpful comments which have improved the presentation of the paper.
D.G. and T.O.~were supported by the European Research Council (grant no.~864138 ``SingStochDispDyn"). 
T.O.~was also supported 
 by the EPSRC 
Mathematical Sciences
Small Grant  (grant no.~EP/Y033507/1).
\end{ackno}

\end{document}